\documentclass[letter,11pt]{article}

\newcounter{conj}

\usepackage{comment}
\usepackage[linktocpage=true]{hyperref}
\usepackage[margin=1.2 in, top=1 in, bottom= 1.2 in]{geometry}

\usepackage[english]{babel}
\usepackage{amsfonts}
\usepackage{mathrsfs}
\usepackage[mathscr]{euscript}
\usepackage{bbm}
\usepackage{latexsym}
\usepackage{math dots}
\usepackage{amssymb}
\usepackage{mathtools}
\usepackage{appendix}

\usepackage{enumitem}
\setlist{nolistsep}
\usepackage{amsthm}
\usepackage{relsize}
\usepackage{nicefrac}
\usepackage[capitalize]{cleveref}

\newtheoremstyle{plain}{3mm}{3mm}{\slshape}{}{\bfseries}{.}{.5em}{}
\newtheoremstyle{definition}{2mm}{2mm}{}{}{\bfseries}{.}{.5em}{}
\theoremstyle{plain}
	
\newtheorem{theorem}{Theorem}

\newtheorem{conjecture}[conj]{Conjecture}

\theoremstyle{definition}

\theoremstyle{plain}
\newtheorem*{namedthm}{\namedthmname}
\newcounter{namedthm}
	

\newcommand{\C}{\mathscr{C}}

\title{On sums, products and slices of some Cantor sets}
\author{Aritro Pathak}
\date{}

\begin{document}

\maketitle

\begin{abstract}

     The questions of the measure and finding open intervals in certain sets of sums and products of elements of the middle third Cantor set $\C$ (or a variant of it), have generated considerable interest recently. A broad general framework that makes it possible to deal with these questions was outlined by Astels. The question of finding the measure of $\C \cdot \C$ was considered recently in an article by Athreya, Reznick and Tyson. Astels' methods apply to product sets upon considering a generalized logarithmic Cantor set, while Athreya, Reznick and Tyson's methods become difficult in dealing with sums or products of $m$'th powers as $m$ becomes large. With a new elementary dynamical technique, we can deal with both the questions of sums and products in a satisfactory way, and the proofs are of the same level of complexity as that of an elementary proof of the fact that $\C + \C=[0,2]$. Further, some observations are made on the question of intersections of one central Cantor set with an affine image of another central Cantor set both with arbitrary parameters.
    
 \end{abstract}

\section{Introduction.}

\ \ \  The study of sums and products of Cantor sets have been motivated by problems in dynamical systems, harmonic analysis, geometric measure theory and number theory; see \cite{Mult, Takahashi} and the references therein.
 
 Let $\C$ be the usual middle third Cantor set:

\[ \C=\Big\{ \sum\limits_{i=1}^{\infty}\frac{t_{i}}{3^{i}}: t_{i} \in \{0,2\} \Big\}  \]
 
 The sets $\mathscr{C}+ \mathscr{C}$ and $\mathscr{C}-\mathscr{C}$ have been well studied  \cite{Randolph, Utz, Majumdar,Halmos}.

One of the folklore proofs that shows $\frac{1}{2}(\C + \C) =[0,1]$ considers the ternary expansion of any real number $x\in [0,1]$, and shows constructively that there are always two elements belonging to $\C$, whose average gives us $x$. Another elementary argument considers the set $\C \times \C \in \mathbb{R}^{2}$ and the lines $x+y=a$ with $a \in [0,2]$, and uses compactness arguments to conclude that this line always contains a point of $\C \times \C$. In a recent article by Jayadev Athreya, Bruce Reznick, Jeremy Tyson \cite{Cantor}, the measure of the set $\C \cdot \C$ was studied, and some obvious gaps were noted within $[0,1]$ which could not be covered by this set. \footnote{They also describe the set $\frac{\mathscr{C}}{\mathscr{C}}:= \big\{\frac{u}{v}: u,v \in \mathscr{C}, v \neq 0 \big\}$.}

In their article they also pose the apparent following conjecture \cite{Cantor}:. 

\begin{conjecture} 
Every  $u\in [0,1]$ can be  written as $x_{1}^2 + x_{2}^2 +x_{3}^2 +x_{4}^2$, with  $x_{i}\in \C$.
\end{conjecture}

It turns out this conjecture is true, and the proof follows easily from earlier work of Newhouse \cite{Newhouse} and  Astels \cite{Astels}. Theorem 2.2, Part 1\footnote{This theorem follows from theorem 2.3 of Astels' paper which is actually a result by Newhouse \cite{Newhouse}.}, from Astels' paper, or the more general Theorem 2.4 Part 1, is enough to prove the above conjecture. \footnote{A more recent paper \cite{Sum} proves this conjecture from \cite{Cantor}, and there is more recent work \cite{Guo} on the general question of the sum of powers, both of which extend the original methods of \cite{Cantor}, but both of which also seem unaware of the work of Astels, which fundamentally work very well for such Waring-Hilbert type questions, as we explain.}

Astels' approach works well for the question of finding intervals in the set $\big\{ x_1 ^2 +x_2 ^2 +x_3 ^2 +x_4 ^2: x_1,x_2,x_3,x_4 \in \C \big\}$, and pursuing his strategy for the second problem we have to construct certain generalized logarithmic Cantor sets which reduces the multiplicative problem to an additive problem. In Athreya, Resnick and Tyson's paper this question was addressed differently, and then the topological structure of the set $\C \cdot \C$ was further studied in a different paper \cite{Mult}.  There has been more recent work on the Lebesgue measure of the product of the Cantor sets \cite{Marchese}, which extends the original methods of \cite{Cantor}.

The main ingredient in \cite{Cantor} is the following simple theorem, which had earlier appeared in \cite{Cabrelli2}, and which was again employed for use in \cite{Sum}. 

\begin{theorem}
   Suppose $\{K_i\} \subset \mathbb{R}$ are nonempty compact  sets such that $K_{1}\supset K_2 \supset K_3 \supset ..$, and $K= \cap K_{i}$. If $F: \mathbb{R}^{m} \to \mathbb{R}$ is continuous, then $F(K^{m})=\cap F(K_{i}^{m})$.
\end{theorem}

It's not difficult to see, by considering $\C$ and the usual definition as the intersection of infinitely many nested compact sets, (at the $k+1$'st stage the set $K_{k+1}$ is obtained by deleting $2^{k}$ many open intervals each of length $1/3^{k}$ from the set $K_{k}$ at the $k$'th stage ), and the function $F(x,y)=x+y$, that $F(K_{i}^{2})=[0,2], \forall i\in \mathbb{N}$, and thus $\C +\C=[0,2]$. Even just considering the function $F(x,y)=xy$, and considering the sets $F(K_{i}^{2})$ for all $i$, an estimate of the size of the set of common intersection is tedious to determine.

One could use this same technique and consider $F(x_{1},x_{2},..,x_{t})=x_{1}^{m}+x_{2}^{m}+..+x_{t}^{m}$ for any positive integers $t,m$, but then it becomes progressively more difficult to keep track of the intersections of the sets $F(K_{i}^{m})$ for all $i$, as $m$ is made large. For $m=2$ and $t=4$ this has been done in \cite{Sum}, but this is where the generalized construction of a Cantor set by Astels, and his theorems prove to be effective in order to get results.

By considering only two square terms in the statement of the above conjecture, we'll have to leave out the interval $(2/9,4/9)$, with three terms we will have to leave out $(3/9,4/9)$ while with four terms we would get an entire open interval $[0,2^{2}]$. It is also clear that by considering $m'$th powers, we would need to consider $2^{m}$ many terms to get all of $[0,2^{m}]$, and with lesser number of terms, we would have to leave out the obvious gaps. 

These questions go back to the celebrated conjecture of Palis that either the sum of the Hausdorff dimensions of typical pairs of dynamically defined Cantor sets is not greater than 1, or the sum of two such Cantor sets contain some open interval. This conjecture was resolved by Yoccoz and Moreira \cite{Yoccoz}. \footnote{When constructing a general dynamically defined Cantor set, at the first stage of dissection we can delete more than one open interval, and continue therein. Our approach would likely need minor modifications depending on the location and size of the open intervals being deleted.}

In our specific context with the central Cantor sets, for a given exponent $m$,we can ask how few terms $t$ in the sum of $m$'th powers one can consider to get an open interval that is contained in their sum. Theorem 2.4 Part 1 of Astels, modified slightly, would give us a good answer here.

In section 2 and 3 of this paper, we consider the $m$'th power of a middle third Cantor set (with $m$ being a positive integer) which is a set of the form,

\begin{equation*}
\C^{(m)}:= \{x^{m}:x\in \C \}
\end{equation*}

with $\C$ being the middle third Cantor set. This generalizes analogously for any central Cantor set.

In Astels paper, an outline is made of the construction of a generalized Cantor set.\footnote{By a generalized Cantor set we mean the set constructed by Astels in full generality; in this paper all results are for the specific case of $\C^{(m)}$.} Consider the $k$'th stage of the canonical construction of such a Cantor set, where we have $2^{k}$ many open intervals, $A_{i}^{(k)}, i=1,2,3,..,2^{k}$. From each one of these, we delete an interval $O_{i}^{(k+1)}$, and are left with two remaining intervals $A_{i0}^{(k+1)}$ and $A_{i1}^{(k+1)}$. For each $i$ ranging from $1$ to $2^{k}$, we get two new sets as above, and along with suitable relabelling, we assign a superscript $(k+1)$ to each of these resulting $2^{k+1}$ new intervals obtained for the $(k+1)'$th stage of construction. We set 

\[\tau(A_{i}^{(k)})=\text{min}\Bigg(\frac{|A^{(k+1)}_{i0}|}{|O_{i}^{(k+1)}|},\frac{|A_{i1}^{(k+1)}|}{|O_{i}^{(k+1)}|}\Bigg)  \]

The thickness of the Cantor set is defined as: $\tau(\C^)=\underset{A_{i}^{(k)}}{\inf}\tau(A^{(k)}_{i}) $. Also consider the number $\gamma(\C^{(m)})=\frac{\tau(\C^{(m)})}{1+\tau(C^{(m)})}$. Astel's theorem implies that with $t_{m}:=\lceil \frac{1}{\gamma(C^{m})} \rceil$ many terms, we would have an open interval in the sumset $\C^{(m)}+...+\C^{(m)} ( \text{added} \  t_{m} \ \text{times})$. 

When considering the set $\C^{(m)}$, $\tau(\C^{(m)})=\frac{1}{2^{m}-1}, \gamma(\C^{(m)})=\frac{1}{2^{m}}$, and thus with $2^{m}$ many terms, we would contain an interval (in fact, we can cover all of $[0,2^{m}]$). 

However, if we were to only restrict to $\C^{(m)} \cap [(\frac{2}{3})^{m},1]$, then we again have a generalized Cantor set $\C'^{(m)}=\C^{(m)} \cap [(\frac{2}{3})^{m},1]$, and in that case, the thickness increases, and the $\gamma$ factor becomes $\frac{7^{m}-6^{m}}{8^{m}-6^{m}}$, and thus for large $m$, we could do with approximately $ \lceil (\frac{8}{7})^{m} \rceil $ many terms. In fact, as the conditions of part 2 of Astels' theorem are satisfied, we would have that $\C'^{(m)}+...+\C'^{(m)} ( \text{added} \  t_{m} \ \text{times}) $ is an interval that has measure approximately $\lceil (\frac{8}{7})^{m} \rceil (1- (\frac{2}{3})^{m})$. 
Clearly if we considered the set $\C^{(m)}\cap [(\frac{8}{9})^{m},1]$, we would get an interval of smaller measure, but with even fewer number of terms. 

As mentioned earlier however, for Astels' methods to apply for the question of products of Cantor sets, one would need to first consider certain generalized logarithmic Cantor sets to reduce the problem to a corresponding question on sums of such generalised Cantor sets. On the other hand, as mentioned earlier, the methods of \cite{Cantor},\cite{Mult} were considered for the problem of considering products of two Cantor sets, but become tedious on the question of the sums of $\C^{(m)}$.

In this paper, we use an elementary dynamical argument that gives us an open interval contained within the set $\C \cdot \C \cdot \C \cdot \C$, and also the same basic argument gives us an open interval within the set $\C^{(m)}$ summed sufficiently many times.

The method is short and different from other recent methods used in these problems. We state our two basic illustrative results below.

\begin{theorem}\label{thm:prod}
The set $[(\frac{8}{9})^{3},\frac{8}{9}]\subset \C \cdot \C \cdot \C \cdot \C$, and each $x\in [(\frac{8}{9})^{3},\frac{8}{9}]$ can be written as $x=c_1 \cdot c_2 \cdot c_3 \cdot c_4$ where each $c_i \in \C \cap [\frac{8}{9},1]$ for $i \in \{1,2,3,4\}$.
\end{theorem}

\begin{theorem}
Let $m\geq 1$ be an integer. Consider $t_{m}=2\cdot \lceil (\frac{3}{2})^{m-1} \rceil$, and write $s_{m}=\frac{1}{2}\cdot t_m$. Then the set $\C^{(m)}+\C^{(m)}+..+\C^{(m)}$, where $\C^{(m)}$ is added $t_{m}$ many times, contains the interval $I=[(s_m+1)(\frac{2}{3})^{m}+(s_m-1),(s_m-1)(\frac{2}{3})^{m}+(s_m+1)]$ of measure $2(1-(\frac{2}{3})^{m})$. Each $x\in I$ can be written in the form $c_1 ^{m}+\dots +c_{t_m}^{m}$ where each $c_1,\dots, c_{t_m} \in \C \cap [\frac{2}{3},1]$.
\end{theorem}

We do not find the largest possible union of open intervals in the sums and products of Cantor sets with this method, but demonstrate that the question of finding an open interval is of a similar level of complexity as our elementary proof that shows $\C+\C=[0,2]$. Further, with our method, we demonstrate an algorithm that constructs Cauchy sequences of elements of $\C$ that converge to the elements of the Cantor set that make up the product or sum expression.

This is a slight variation of the method of an earlier paper \cite{Aritro} that gives us $\C +\C =[0,2]$, which we restate in the Appendix. Although the result is well known, the method serves as a basic illustration of the ideas used in Section 2 and 3 of this paper. In Section 2, we outline the proof for the case of the product $\C \cdot \C \cdot \C \cdot \C$. In Section 3, we outline our argument for the sum of the sets $C^{(m)}$.

The methods of Section 2 of this paper will be seen to be applicable for questions on arbitrary central $\C_{\lambda}$ Cantor sets, ($C_{\lambda}$ is the usual central Cantor set where at each $k$'th stage, a middle $(1-2\lambda)$ fraction open interval is cut out from the intervals existing at the $(k-1)$'th stage.) and the question of finding open intervals in product sets such as
$\C_{\lambda_1} \cdot \C_{\lambda_2}\cdot..\cdot \C_{\lambda_m}$. This question was raised as an open question in the concluding section of recent work by Takahashi \cite{Takahashi}, where the study of products of Cantor sets was motivated by certain dynamical questions. In a manuscript in preparation, we are considering these questions.

Again, the methods of Astels considering logarithmic Cantor sets should in principle give us answers here (or Theorem 1, however difficult, would still work in principle), but the methods of Section 2 of this paper will generalize without more difficulty or increase of complexity.
For the corresponding question on sumsets, there has been earlier work by Cabrelli, Hare and Molter \cite{Cabrelli1,Cabrelli2}.
\footnote{Further, our methods for proving Theorem 3 can be easily extended for the case where $m$ is any real positive number, not necessarily an integer; one key ingredient in the bounds used for Theorem 3 is the binomial theorem to first approximation, and that can be easily adopted for the case of $m$ not being an integer. The method can also be easily extended for more non-trivial polynomial or rational expressions involving sums and products of Cantor Sets. }

When restricting to two Cantor sets $C_a,C_b$, the dimension of their sumset is equivalent to the dimension of the projection of the set $C_a \times C_b \subset [0,1]^{2}$ onto the diagonal line $x=y$ in the plane.\footnote{All the results on sumsets of  Cantor sets we have considered thus far can be equivalently framed as projection results on hyperplanes in $R^{m}$ for $m\geq 2$.} When the parameters $a,b$ satisfy the assumption that $\log a /\log b$ is irrational, it was shown by Peres and Shmerkin \cite{Shmerkin} that the dimension of the sumset is ``large": specifically that $dim(C_a +C_b)=\text{min}(dim(C_a)+dim(C_b),1)$. A natural dual question is that of the dimension of the slices of $C_a \times C_b$ by straight lines, and under the irrationality assumption a conjecture is made in the concluding section of Peres, Shmerkin's paper.

In full generality, without the irrationality assumption, we end with brief observations on this slicing question, in Section 4, which we intend to develop further in future work, as outlined there.



\section{Open intervals in the product of Cantor sets.}

We demonstrate that the set $[(\frac{8}{9})^{3},\frac{8}{9}]\subset \C \cdot \C \cdot \C \cdot \C$. The arguments of the proof can be repeated in a straightforward way to locate other closed intervals within $\C \cdot \C \cdot \C \cdot \C$, which is something we don't do here. The method used here is a slight variation on the method of the proof in the Appendix. As mentioned earlier, the analysis of the measure and structure of the set $\C \cdot \C$ has already been carried out in other papers \cite{Cantor}, \cite{Mult}. While our technique doesn't seem to enable us to state results for just the product set $\C \cdot \C$, it gives an alternate elementary and dynamical way to tackle a slightly weaker problem, and the same techniques applied to this problem are later made to work in a straightforward way on the problem of sums of $m'th$ powers of Cantor set elements. 

The set $[(\frac{8}{9})^{3},\frac{8}{9}]\subset \C \cdot \C \cdot \C \cdot \C$.

\begin{proof}

(Of \cref{thm:prod}:) Consider any point $x\in [(8/9)^{3} , 8/9]$ and take $a_1=b_{1}=8/9$, and $c_{1}=d_{1}=1$. Thus the entire dynamics is within the interval $[(8/9),1]$. Similar to the proof in the Appendix, the idea is to construct four different Cauchy sequences $a_{k},b_{k},c_{k},d_{k}|_{k=1}^{\infty}$, with each of $a_{k},b_{k},c_{k},d_{k}$ (for $k$ from $1$ to $\infty$) being elements of $\C$, so that in the limit we have four points, $a_{\infty}, b_{\infty}, c_{\infty}, d_{\infty}$ belonging to $\C$ with $x=a_{\infty}\cdot b_{\infty}\cdot c_{\infty}\cdot d_{\infty}$. We also write the product $p_{k}:=a_{k}\cdot b_{k} \cdot c_{k}\cdot d_{k}$ for every positive integer $k$, and the difference $\Delta_{k}:=x-p_{k}$.

If $x= (8/9)^3, (8/9)^{2}, \  \text{or} \ (8/9)$ we are done. Note that $p_{1}=(8/9)^{2}$.

\begin{enumerate}

\item First assume that $(8/9)^{3}<x< (8/9)^{2}$. Consider the value $x_{0} \in ((8/9),1)$ so that $x=a_{1}\cdot b_{1} \cdot x_{0} \cdot d_{1}=(8/9)^{2}\cdot x_{0}$. Unless $x_{0}\in \C $ in which case we are done, consider the open interval that $x_{0}$ falls in the interior or the boundary of, in the process of iteratively constructing $\C \cap [8/9,1]$ and take $c_{2}$ to be the left end point of this interval, and take $a_{2}=a_{1}, b_{2}=b_{1}, d_{2}=d_{1}$. Then $\Delta_{2}= (8/9)^{2}(x_{0}-c_{2})<\frac{1}{3^{3}}$, since $(x_{0}-c_{2})$ is bounded by the length of the cut out interval under consideration, and the biggest possible such gap within $[8/9,1]$ has length $1/3^{3}$. Consider the unique $k_{2}\geq 3$ so that $\frac{1}{3^{k_{2}+1}}\leq \Delta_{2}=x-p_{2}<\frac{1}{3^{k_{2}}}$.

The aim is to iteratively increase the sequences $a_{k},b_{k}$ and decrease the sequences $c_{k},d_{k}$ in a controlled way so that as $k\to \infty$, we have $\Delta_k \to 0$. Notice that we can hypothetically next define $c_{3}$ so that the difference $c_2 -c_3 = 2/3^{\tilde{k_{2}}+1}$, with $\tilde{k_2}\geq k_2$ and $\Tilde{k_2}$ being an integer, while similarly we can define $d_{3}$ so that $d_2 -d_3=2/3^{3}$ or lesser, and further similarly it is possible to define $a_{3},b_{3}$ so that the differences follow the method of proof in the Appendix.

So at this next $3$'rd stage of iteration, we try to increase either one or both of $a_{2}, b_{2}$, in order to decrease the difference between $x$ and $p_{3}$ from the value $\Delta_{2}$. If we only take $a_{3}=a_{2}+ \frac{2}{3^{k_{2}+1}}$, and keep $b_{3}=b_{2},c_{3}=c_{2}, d_{3}=d_{2}$, then $(\frac{8}{9})^{3}\cdot (\frac{2}{3^{k_{2}+1}})< p_{3}-p_{2}=b_{2}\cdot c_{2}\cdot d_{2}\cdot \frac{2}{3^{k_{2}+1}}< \frac{2}{3^{k_{2}+1}}$. Call the number $\alpha\approx (\frac{8}{9})^{3}(\approx 0.70)$. Thus, $\Delta_{3}=x-p_{3}=(x-p_{2})+(p_{2}-p_{3})$. Recall that $\frac{1}{3^{k_{2}+1}}\leq (x-p_{2})< \frac{1}{3^{k_{2}}}$. Thus we have, $-\frac{1}{3^{k_{2}+1}}=\frac{1}{3^{k_{2}+1}} -\frac{2}{3^{k_{2}+1}}< \Delta_{3}<\frac{1}{3^{k_{2}}}-\frac{2\alpha}{3^{k_{2}+1}}$. In this process above, the upper bound is greater than $\frac{1}{3^{k_{2}+1}}$, and if in this process, $\Delta_{3}$ is actually  less than $\frac{1}{3^{k_{2}+1}}$ we are done, otherwise if in the process $\Delta_{3}$ is greater than $\frac{1}{3^{k_{2}+1}}$, i.e. now $\frac{1}{3^{k_{2}+1}}<\Delta_{3}<\frac{1}{3^{k_{2}}}-\frac{2\alpha}{3^{k_{2}+1}}$, then we take $k_{3}=k_{2}$,  $b_{4}=b_{3}+\frac{2}{3^{k_{3}+1}}, a_{4}=a_{3}, c_{4}=c_{3}, d_{4}=d_{3}$. Again, $ \frac{2\alpha}{3^{k_{3}+1}}<p_{4}-p_{3}< \frac{2}{3^{k_{3}+1}}$ and thus \\ $ -\frac{1}{3^{k_{3}+1}}<\Delta_{4}=(x-p_{4})=\Delta_{3}+(p_{3}-p_{4})< \frac{1}{3^{k_{3}}}-\frac{4\alpha}{3^{k_{3}+1}}<\frac{1}{3^{k_{3}+1}}$.

Thus the absolute magnitude of the error term $\Delta_{4}$ is now bounded by $\frac{1}{3^{k_{3}+1}}$. Choose the unique $k_{4}\geq k_{3}+1$ so that now $\frac{1}{3^{k_{4}+1}}\leq |\Delta_{4}|< \frac{1}{3^{k_{4}}}$. Now it is clear that we can mimic the  proof of the previous step and iterate the argument, and further decrease the error term in the next iterations so that successively we get $\frac{1}{3^{k_{l}+1}}<|\Delta_{l}|<\frac{1}{3^{k_{l}}}$ where necessarily $k_{l}\geq k_{l-1}+1$. The $\alpha$ factor remains the same as we progress along the iteration, and the bounds work in the same manner, as $k\to \infty$. Thus clearly in the limit we find four elements $a_{\infty},b_{\infty},c_{\infty},d_{\infty} \in \C$ so that $x=a_{\infty}\cdot b_{\infty}\cdot c_{\infty}\cdot d_{\infty}$.

\bigskip 

\item The argument for the case where $(8/9)^{2}<x<(8/9)$ proceeds in the similar manner; instead of decreasing $c_{1}$ from the value $1$, we instead increase the value $a_{1}$ from $(8/9)$ to the right endpoint of the cut out interval where $x_{0}$ is located, where similar to the previous case, we define $x_{0}$ to be the number so that $x=x_{0}\cdot b_{1}\cdot c_{1}\cdot d_{1}$. The $\alpha$ factor of course remains the same as before, and then the argument from the previous case can thus be used here as well.

\end{enumerate}
\end{proof}

As mentioned before, this set $[(8/9)^{3},(8/9)]$ can be easily expanded by similar other choices of initial values $a_{1},b_{1},c_{1},d_{1}$. In particular we can find smaller open intervals whose right end points are arbitrarily close to the value of 1, such as $\C \cap [\frac{3^{k}-1}{3^{k}},1]$, for $k\geq 3$. In these cases, the $\alpha$ factor only goes closer to 1, and the algorithm runs as before. \footnote{However, if we were to consider $\C \cap [2/9,3/9]$, then the corresponding $\alpha$ factor becomes much smaller, and in that case we would require the product of more terms for the argument to work.}

As mentioned in the introduction, this is a construction that can be made to work for the corresponding question of products of different central Cantor sets, which was raised in the concluding section of \cite{Takahashi}. With a more intricate version of essentially the same argument used here, in a future manuscript we are considering products of different central Cantor sets, of the form $C_{\lambda_1}\cdot C_{\lambda_2}\dots C_{\lambda_k}$ for different integers $k$ and parameter sets $\{\lambda_1\dots \lambda_k \}$ to find open intervals.

\section{Sums of $\C^{(m)}$.}

Now we finally come to the question of sums of $m'$th powers, with $m\geq 2$, and finding open intervals within them. This is a problem where Astels' more involved but general methods gives stronger results, but where the methods of \cite{Cantor},\cite{Mult} become difficult as $m$ grows larger.



\begin{proof}

(of Theorem 3:) Consider the sequences $a^{(1)},a^{(2)},...,a^{(s_{m})}$, $b^{(1)},b^{(2)},..,b^{(s_{m})}$ with  $a^{(1)}_1=a^{(2)}_1=...=a^{(s_{m})}_1=\frac{2}{3}$, and $b^{(1)}=b^{(2)}=...=b^{(s_{m})}=1$. 

In this case, for each integer $k\geq 1$, we define 

\[S_{k}=(a_{k}^{(1)})^{m}+..+(a_{k}^{(s_{m})})^{m}+(b_{k}^{(1)})^{m}+..+(b_{k}^{(s_{m})})^{m}\]

Note that $S_{1}=s_m((\frac{2}{3})^{m}+1)$.

Consider any $x\in I$. If $x=S_{1}$ then of course we are done. Consider w.l.o.g, $x< S_{1}$. (The case for $x> S_{1}$ proceeds in an exactly similar way.)

In this case we consider the value $x_{0}\in [2/3,1]$ so that $x=(a_{1}^{(1)})^{m}+(a_{1}^{(2)})^{m}+..+(a_{1}^{(s_{m})})^{m} +x_{0}^{m}+(b_{1}^{(2)})^{m}+..+(b_{1}^{(s_{m})})^{m} $. As before, if $x_0\in \C$ we are done, otherwise consider the open interval, of length some $1/3^{k_{1}}$ with $k_{1}\geq 2$, that $x_{0}$ falls in the interior or the boundary of, in the process of iteratively constructing $\C$, and consider the left end point of this cut interval, and call it $b_{2}^{(1)}$, while keeping all other sequences constant between $k=1$ and $k=2$ stages of the sequence. Again consider the error term $\Delta_{k}= x - S_{k}$. In this situation, the difference $\Delta_{2}=x-S_{2}$ is clearly bounded from above by $(b_{2}^{(1)}+ \frac{1}{3^{k_{1}}})^{m} - (b_{2}^{(1)})^{m}=\frac{1}{3^{k_{1}}}\cdot (...)$ where there are exactly $m$ terms in the bracket above, each of which is bounded from above by 1 (and bounded from below by $(\frac{2}{3})^{m-1}$). Consider the unique integer $k_{2}\geq k_{1}$ so that $\frac{m}{3^{k_{2}+1}}\leq \Delta_{2}< \frac{m}{3^{k_{2}}}$. Note\footnote{similar to the previous section} that here it is possible to decrease $b_{2}^{(1)}$ by an amount $\frac{2}{3^{\tilde{k_{2}}+1}}$ where the integer $\tilde{k_2}\geq k_{2}$.

Now we increase $a_{2}^{(1)}$, and consider, $a_{3}^{(1)}=a_{2}^{(1)}+\frac{2}{3^{k_{2}+1}}$, while keeping all the terms constant when passing from the $k=2$ to $k=3$ stages of the sequence. The increment in the value of $S_{3}$ from the value $S_{2}$ is given by $(a_{2}^{(1)}+\frac{2}{3^{k_{2}+1}})^{m}-(a_{2}^{(1)})^{m}=\frac{2}{3^{k_{2}+1}}\cdot (...)$ where again there are $m$ terms in the bracket each of which is bounded from above by $1$ and bounded from below by $(\frac{2}{3})^{m-1}$. The objective here is to make the error lower so that $|\Delta_{3}|<\frac{m}{3^{k_{2}+1}}$. 

It should now be clear from the above statements that at least one of the $a^{(i)}$, for $i$ varying from 1 to $s_{m}$, needs to be increased, and in the extreme case, we might possibly end up needing $u_{m}$ many terms where $u_{m}$ is the minimum integer so that we have $u_{m}\cdot \frac{2m}{3^{k_{2}+1}}(\frac{2}{3})^{m-1}>(\frac{m}{3^{k_{2}}}-\frac{m}{3^{k_{2}+1}})=\frac{2m}{3^{k_{2}+1}}$, and thus we exactly have $u_{m}=\lceil (\frac{3}{2})^{m-1} \rceil=s_{m}$. 

In fact, in one extreme case, when $\Delta_2=\frac{m}{3^{k_{2}+1}}$, then it is sufficient to increase exactly one of the $a^{l}$ terms for any $1\leq l \leq s_m$, and w.l.o.g letting $l=1$, we can have that $\frac{m}{3^{k_{2}+1}} -\big(\frac{2}{3}\big)^{m-1}\frac{2m}{3^{k_{2}+1}} > \Delta_3 > \frac{m}{3^{k_{2}+1}}- \frac{2m}{3^{k_{2}+1}}=-\frac{m}{3^{k_{2}+1}}$, and thus $|\Delta_{3}|<\frac{m}{3^{k_{2}+1}}$. In the other extreme case, when we have $\Delta_2=\frac{m}{3^{k_{2}}}$, then we use $\lceil (\frac{3}{2})^{m-1} \rceil$ many terms as noted above to ensure that $\Delta_3$ is brought below the value of $\frac{m}{3^{k_2 +1}}$. Along with the above, since the magnitude of each of the individual decrements are bounded from above by $\frac{2m}{3^{k_2 +1}}$, for any value of $\frac{m}{3^{k_{2}+1}}\leq \Delta_2 \leq \frac{m}{3^{k_{2}}}$, by using at most $s_m$ many of the $a^{(l)}$ terms, we can ensure that $|\Delta_{3}|<\frac{m}{3^{k_{2}+1}}$.

In the subsequent stages, we might have to decrease upto $s_{m}$ of the $b_{k}^{(i)}$ terms where $i$ ranges from 1 to $s_{m}$, or at a different step increase upto $s_{m}$ of the $a_{k}^{(i)}$ terms, in order to ensure that the absolute values of the errors terms $\Delta_{k}$ are bounded at each succesive stage by a higher exponent of $1/3$. Thus in total we need $2s_{m}=t_{m}$ many terms. 

The proof for the case where $x>S_{1}$ also follows in an identical way; in the beginning, we increase $a_{1}^{(1)}$ to $a_{2}^{(1)}$ by an appropriate amount, and then enforce our dynamical argument. In the end we find $t_{m}$ many limit points so that the sum of their $m$
th powers is exactly the value $x$.  
\end{proof}

Again we should remark that for $k\geq 3$, by focusing on a smaller interval $[1-\frac{1}{3^{k}},1]$, with this method we should get, with $2\cdot \lceil (\frac{3^{k}}{3^{k}-1})^{m-1} \rceil$ terms, an open interval of length $2\cdot(1-(\frac{3^{k}-1}{3^{k}})^{m})$.

\section{Slices of Cartesian products of two central Cantor sets.}

\ \ \ \ Until now, all results on sums of generalized or central Cantor sets can be interpreted as projection theorems on  hyperplanes in $\mathbb{R}^{m}$. Here we briefly discuss some aspects of the dual slicing question for central Cantor sets; the intersection of one central Cantor set with the affine image of another central Cantor set.

In the concluding section of Peres, Shmerkin's paper \cite{Shmerkin} on the dimension of the sumset of two central Cantor sets $C_a, C_b$ with the assumption of $(\log a/\log b)$ being irrational, it is conjectured that for every $\theta \in (0,\pi) \backslash \{ \pi/2\}$ there is a positive measure set of lines in the direction $\theta$ each of which intersects $C_a \times C_b \subset [0,1]^{2}$ in a set of dimension $dim(C_a)+dim(C_b)-1$ whenever $\dim(C_a)+\dim(C_b)>1$.\footnote{By Furstenberg's conjecture, this is the maximum possible such dimension that can be attained by any slice of $C_a\times C_b$ under the irrationality assumption.}

There has been considerable ongoing interest in problems around  the question of slices of product sets $A\times B$ where $A, B\subset[0,1]$ are respectively invariant under the $\times a,\times b$ maps with $(\log a/\log b)$ irrational, building on the celebrated proofs of the Furstenberg slicing theorem by Shmerkin, Wu and later Austin. See \cite{Shmerkin2},\cite{Shmerkin21},\cite{Wu},\cite{Algom},\cite{Yu},\cite{Austin}.

In full generality, we drop the irrationality assumption here, and consider any two arbitrary central Cantor sets $C_a, C_b$, their product $C_a \times C_b$ and their slices by all lines in a specific direction $\theta \in (0,\pi/2)$.\footnote{An analogous statement holds for any $\theta\in (\pi/2, \pi)$.}. The method here is broadly the same as the basic arguments in the beginning of Section 4 of \cite{Wu}. We prove the following\footnote{The author is thankful to Pablo Shmerkin for communication on this question.}:

\begin{theorem}Assume there is a set S of positive measure of lines with slope $m=\tan \theta$ so that each line of this set intersects $C_a \times C_b$ in a set of Hausdorff dimension at least $t$ with $t\leq \text{max}(0, \text{dim}(C_a)+\text{dim}(C_b)-1)$.\footnote{Due to the Marstrand slicing theorem, we can have at most a set of zero Lebesgue measure of lines that intersect $C_a\times C_b$ in Hausdorff dimension $t> \text{max}(0, \text{dim}(C_a)+\text{dim}(C_b)-1)$. } Then for each fixed pair of integers $(n,m)$, there exists a set of positive measure of lines with slope $m \cdot (\frac{a^{n}}{b^{m}})$, each of which also intersects $C_a \times C_b$ in a set whose Hausdorff dimension is at least $t$. 

\end{theorem}

\begin{proof} Separately we show that the argument is true for the cases (i) when $n=m$, (ii)$m=0$, (iii)$n=0$. Cases  (ii) and (iii) are similar and so we illustrate the method only for Case (ii). The general argument follows by iterating these two separate arguments. 

For Case (i), within $[0,1]^{2}\subset \mathbb{R}^{2}$, assume we have $C_a$ on the $x$-axis and $C_b$ on the $y$-axis. In the first level of dissecting $C_a\times C_b$, we have four separate rectangles within $[0,1]^{2}$, equivalent to each other, each of length $a$ along the x-axis and length $b$ along the y axis. It is easy to see that within the set $S$ of lines in direction $\theta$ there is a subset $S_1$ of positive measure which intersects at least one of these four rectangles, say $R_1$, with each line $l\in S_1$ being such that $dim(l\cap R_1)\geq t$. This also guarantees a positive measure set of lines $S'_{1}$ with slope $\tan \theta \cdot (\frac{a}{b})$ each of which intersects $C_a \times C_b \subset [0,1]^{2}$ in dimension at least $t$, because the dimension is clearly invariant under the map $x\to x/a, y\to y/b$ in going from $R_1$ to $[0,1]^{2}$. As we iterate this argument, we get the  statement for all positive integers $n$. On the other hand, consider the set of all lines $S''$ with slope $\tan \theta \cdot(\frac{a}{b})^{-1}$. Within any of the four rectangles, say $R_1$ of $R_1,R_2,R_3,R_4$ in the first stage of dissection, the set of intersections $\{l\cap R_1|l\in S''\}$ is affinely equivalent to the set of intersections of all the lines with slope $\tan\theta$ intersecting $C_a\times C_b$. Thus by hypothesis, there is a positive measure of lines of $S''$ intersecting $R_1$ in dimension at least $t$, hence also intersecting $C_a \times C_b$ in dimension at least $t$. Iterating this argument gives us the result for all negative integers $n$.

For Case (ii) consider the regions $\alpha_1:=A_1\times [0,1],\alpha_2:=A_2\times [0,1]$, where $A_1$ and $A_2$ are the intervals on the $x$-axis that are cut out in the first stage of dissection of $C_a$. We note that under the transformation where $y\mapsto y, x\mapsto \frac{1}{a}x$, the line $l$ given by $y=mx+c \mapsto y=max+c$ which we call $l'$,and the dimension of $l\cap \alpha_1\cap (C_a\times C_b)$ equals the dimension of $l'\cap (C_a\times C_b)$, and a similar statement holds for lines intersecting $\alpha_2$. Now by the same basic arguments of the preceding paragraph, the result follows for Case (ii) and an identical argument gives the result for Case (iii).
\end{proof}

The arguments above easily generalizes to products of attractors of arbitrary regular self-similar IFS's.

In the statement of the previous theorem we could also require equality, and it is easily seen that the proof follows analogously:

\begin{theorem}Assume there is a set S of positive measure of lines with slope $m=\tan \theta$ so that each line of this set intersects $C_a \times C_b$ in a set of Hausdorff dimension exactly $t$ with with $t\leq \text{max}(0, \text{dim}(C_a)+\text{dim}(C_b)-1)$. Then for each fixed pair of integers $(n,m)$, there exists a set of positive measure of lines with slope $m \cdot (\frac{a^{n}}{b^{m}})$, each of which also intersects $C_a \times C_b$ in a set whose Hausdorff dimension is exactly $t$. 

\end{theorem}

Thus for Furstenberg's slicing inequality for the specific case of the product of two different central Cantor sets, if the equality is attained for a specific direction, it is also attained for countably many directions. 

Further, this same basic argument also shows that if in one specific direction of slope $\tan \theta$, there exists a set $S$ of some Hausdorff dimension $0\leq \eta\leq 1$ of lines, each of which intersects $C_a\times C_b$ in positive measure, then there is a countable set of directions with slopes $(\tan \theta)(a^{n}/b^{m})$ as $(n,m)$ varies over all possible pairs of integers, so that in each of these directions there is also a corresponding set of dimension $\eta$ of lines, each line of which intersects $C_a \times C_b$ in positive measure. \footnote{Furthermore, if in a specific direction of slope $\tan \theta$ we have a set $S$ of some Hausdorff dimension $\alpha$ of lines so that each line belonging to $S$ intersects $C_a \times C_b$ in a set of Hausdorff dimension at least ( or equal to) some $\beta$, then in each of the directions with slope $(\tan \theta)(a/b)^{n}$, we have a set of Hausdorff dimension $\alpha$ of lines so that each line of this set intersects $C_a \times C_b$ in a set of Hausdorff dimension at least (or equal to) $\beta$.}

In future work, we intend to further study the set of slices of general sets of the form $C_a\times C_b\subset [0,1]^{2}$; not necessarily satisfying the irrationality assumption. \footnote{One could conjecture that for certain quantifiably large sets of parameters $a,b$, every single slice of the set $C_a \times C_b$ should have Hausdorff dimension at most $\max\{0,\dim(C_a)+\dim(C_b)-1\}$.}

\section{Acknowledgements:} The author is grateful to Larry Guth and Omer Offen for listening to some of the arguments of this paper, to Pablo Shmerkin for useful feedback and pointing out Astels' and Wang, Jiang, Li, Zhao's papers, and also to Boris Hasselblatt for going through two earlier drafts of the manuscript, and his suggestions for improvement of the paper.

\bigskip

\appendix
\section{Sum of two middle third Cantor sets}
We reproduce a proof of the fact that $\frac{1}{2}(\C +\C) =[0,1]$. This is a well known result, but the idea behind this proof, originally from \cite{Aritro}, was used in the methods in Section 2 and 3 of this paper, and below is a simpler exposition of the idea behind those arguments.

We consider any $x\in [0,1]$; unless $x\in \C$ we consider the interval containing $x$ that is cut out from $[0,1]$ at some stage while constructing $\C$. We consider the endpoints of this interval, which both belong to the Cantor set, as two possible candidates whose average would give us $x$. If this average is slightly higher or lower, we keep moving the lower endpoint even lower, or the higher endpoint higher in a  controlled way, and get two Cauchy sequences in the process, so that in the limit we find two elements belonging to the Cantor Set whose average is exactly $x$.

\begin{theorem}\label{thm:zero} 
Every $u\in [0,1]$ is the average of two real numbers each belonging to the Cantor's middle third set.
\end{theorem}

\begin{proof}
Take an arbitrary real number $y\in [0,1]$. Unless $y\in \C$, in the process of creating the Cantor set from $[0,1]$ by deleting the middle thirds, after a finite number ($k_{0}$) of steps, $y$ would belong to an open set that's cut out for the first time. \footnote{Indeed if the ternary expansion of $y$ contains a $1$, then $y$ falls in the interior or on the left boundary of a cut open third corresponding to the first time the $1$ appears in the ternary expansion. Otherwise $y$ contains only $0 \ \text{or} \ 2$ in it's ternary expansion, and then $y\in \C$ and we are done.  The length of the interval cut out at this $k_{0}$'th iteration is $1/3^{k_{0}}$. }

(i) Let the closest end point to $y$ at this stage on the right be $a_{1}\in [0,1]$ at a distance $r_{1}:=|a_{1}-y|$, and that on the left be $b_{1} \in [0,1]$ at a distance $l_{1}:=|b_{1}-y|$ (we have $a_{1},b_{1}\in \textit{C}$).  Consider the unique $k_{1}>0$ so that $1/3^{k_{1}+1}<|l_{1}-r_{1}|\leq 1/3^{k_{1}}$. We have $k_{1}\geq k_{0}$. Consider w.l.o.g $r_{1}\geq l_{1}$. \footnote{When $l_{1} > r_{1}$, the proof follows in an analogous way.}

  (ii)  To the left of $b_{1}$, we further iterate and remove successive middle thirds so that eventually there is a point $b_{2}\in \textit{C} $ to the left of $b_{1}$ with  $l_{2}-l_{1}=2/3^{k_{1}+1}$, where $l_{2}:=|b_{2}-y|$. At this stage, take $a_{2}=a_{1},\text{and} \ r_{2}=r_{1}$.
    
    We have: $ 1/3^{k_{1}+1}-2/3^{k_{1}+1}=-1/3^{k_{1}+1}< r_{2}-l_{2}= ( r_{2}-l_{1})-( l_{2}-l_{1})\leq (1/3^{k_{1}}-2/3^{k_{1}+1})=1/3^{k_{1}+1}$, and so $|r_{2}-l_{2}|\leq 1/3^{k_{1}+1}$. 
    Thus we can find a unique $k_{2}>k_{1}$ so that $1/3^{k_{2}+1}<|r_{2}-l_{2}|\leq 1/3^{k_{2}}$. 
    
    Now we perform steps exactly analogous to the steps (i) and (ii) above. It may happen that at the $k$'th stage, we have $l_{k}>r_{k}$. In this case, corresponding to (ii), we would find a point $a_{k+1}$ to the right of $a_{k}$, while keeping $b_{k+1}=b_{k}$. The sequence $s_{k}=|r_{k}-l_{k}|$
     in the k'th iterative
step is bounded by a higher power of $1/3$, and so $s_{k}\rightarrow 0$ as $k \rightarrow \infty$. 

Here $\{a\}_{i=1}^{\infty}(\{b\}_{i=1}^{\infty})$ is bounded within [0,1], is non-decreasing (non-increasing) and thus converges to a limit point
$a_{\infty}(b_{\infty})$ that also belongs to the Cantor set itself, the Cantor set being closed. In the limit, we thus have within the Cantor set two points that are equidistant
from $y$ (with $r_{\infty} = |a_{\infty}-y|=|b_{\infty}-y|= l_{\infty}$), and this proves our assertion.
\end{proof}

\bigskip

\bigskip

Aritro Pathak, University of Missouri, Math Sciences Building, 810 Rollins St, Columbia, MO 65201

\end{document}